\documentclass[a4paper, 12pt]{amsart}
\usepackage[margin=1in]{geometry}
\usepackage{latexsym, amssymb, amscd, amsfonts, amsthm, amsmath, graphicx, mathabx, mathrsfs}

\newcommand{\Sm}{\operatorname{small}}
\newcommand{\homdeg}{\operatorname{hom-deg}}
\newcommand{\supp}{\operatorname{supp}}
\newcommand{\ab}{\operatorname{ab}}
\newcommand{\pr}{\operatorname{pr}}

\newcommand{\RE}{\operatorname{Re}}
\newcommand{\meas}{\operatorname{meas}}
\newcommand{\Ht}{\operatorname{ht}}

\newcommand{\bR}{\mathbb{R}}

\newcommand{\zed}{\mathbb{Z}}

\newcommand{\id}{\mathrm{id}}

\newcommand{\utau}{\underline{\tau}}
\newcommand{\ux}{\underline{x}}

\newcommand{\umu}{\underline{\mu}}

\newcommand{\fg}{{\mathfrak{g}}}
\newcommand{\fa}{{\mathfrak{a}}}

\newcommand{\E}{\mathbf{E}}

\newtheorem{theorem}{Theorem}

\newtheorem{lemma}[theorem]{Lemma}

\theoremstyle{remark}

\title{The local limit theorem on nilpotent Lie groups}

\author{Robert Hough}
\address[Robert Hough]{Department of Mathematics, Stony Brook University, 100 Nicolls Road, Stony Brook, NY 11794}
\email{robert.hough@stonybrook.edu}


\thanks{Robert Hough is supported by NSF Grant DMS-1712682, ``Probabilistic methods in discrete structures and applications.''}
\thanks{The author thanks Persi Diaconis for his continued interest in the project, and Terence Tao for a helpful conversation.}

\begin{document}

\begin{abstract}
A local limit theorem is proven on connected, simply connected nilpotent Lie groups, for a class of generating measures satisfying a moment condition and a condition on the characteristic function of the abelianization.  The result extends an earlier local limit theorems of Alexopoulos which treated absolutely continuous measures with a continuous density of compact support, and also extends local limit theorems of Breuillard and Diaconis-Hough which treated general measures on the Heisenberg group.
\end{abstract}

\maketitle

\section{Introduction}

Let $G$ be a connected, simply connected nilpotent Lie group.  Alexopoulos \cite{A02a} studied the repeated convolution $\varphi^{*N}$ of a continuous, compactly supported probability density on $G$, approximating the density pointwise with the corresponding heat kernel.  See also \cite{A02b} for related estimates on discrete nilpotent groups and \cite{A02c} for estimates related to the heat kernel. In \cite{B04} and \cite{B05}, Breuillard obtained local limit theorems on the Heisenberg group for general centered measures of compact support.  His theorems were refined by the author and Diaconis in \cite{DH15} weakening the technical assumptions and obtaining the optimal rate. This article extends the method of \cite{DH15} to obtain a local limit theorem with an optimal rate for general Borel probability driving measures on a connected, simply connected nilpotent Lie group, subject to mild technical conditions.

Given $f \in C_c(G)$ and $g \in G$, let $L_g f$ and $R_g f$ denote the left and right translation of $f$ by $g$. For a Borel probability measure $\mu$ on $G$ with finite second homogeneous moments, let $L_\mu$ be the corresponding sub-Laplacian and let $u_t$ be the fundamental solution to the heat equation
$
  \left(\frac{\partial}{\partial t} + L_\mu\right) u_t = 0,$ $t>0$.
\begin{theorem}\label{main_theorem}
 Let $G$ be a connected, simply connected nilpotent Lie group of step $s$ and homogeneous degree $D$.  Let $\mu$ be a Borel probability measure with projection $\mu_{\ab}$ to the abelianization $G_{\ab} = G/[G,G]$ satisfying
 \begin{enumerate}
  \item [i.] (Centered) $\mu_{\ab}$ is mean zero.
  \item [ii.] (Cram\'{e}r) There is a constant $c>0$ and a neighborhood $U$ of 0 in $\hat{G}_{\ab}$ containing 0 such that the characteristic function $\hat{\mu}_{\ab}$ satisfies
  \begin{equation}
   \sup_{\xi \in U^c} \left|\hat{\mu}_{\ab}(\xi)\right| < 1-c.
  \end{equation}
 \end{enumerate}
For all $A > 0$, if $\mu$ has $O_{A,D}(1)$ finite homogeneous moments, then 
uniformly in $g, h \in G$, for all Lipschitz $f \in C_c(G)$,  as $N \to \infty$,
\begin{equation}
 \left\langle L_gR_h f, \mu^{*N}\right \rangle = \left\langle L_gR_h f, u_N\right \rangle + O_\mu\left(\|f\|_1N^{-\frac{D+1}{2}} \right) + O_{\mu, A, f}\left( N^{-A}\right).
\end{equation}
\end{theorem}
\noindent
The dependence on $f$ in the second error term may be controlled in terms of the maximum distance of $\supp f$ from the identity.

The rate is optimal, as may be seen by projecting to the abelianization. The optimal rate does not hold without some decay condition on the characteristic function of the abelianization, although a weaker condition than Cram\'{e}r would suffice. A different limit is obtained in the lattice case, again with optimal rate, by \cite{A02b}.  In \cite{DH15} it is shown that on the the Heisenberg group, the limit statement without a rate can be obtained with the Cram\'{e}r condition replaced with $\left|\hat{\mu}_{\ab}(\xi)\right| \neq 1$ if $\xi \neq 0$; the corresponding statement on a general nilpotent group is currently open.  

\subsection{Discussion of method, and possible extensions}
Theorem \ref{main_theorem} is of the type proved by Breuillard \cite{B04}, \cite{B05} in which an arbitrary translation to the test function is permitted on the left and right.  The proof given there used the representation theory of the real Heisenberg group.  In \cite{B05} Breuillard writes that he expects his analysis to carry through to general Heisenberg groups, but that new methods would need to be developed to handle the higher step cases treated here.   Alexopoulos's theorems hold in the greater generality of groups of polynomial volume growth.  The proofs first establish the results in the connected, simply connected nilpotent case using time domain PDE methods.  It is of interest to obtain the local theorem for general measures in this full generality. 

Theorem \ref{main_theorem} is proved via harmonic analysis on the Lie algebra. At phases much larger than the scale of the distribution, the i.i.d. nature of the increments of the walk is used with a rearrangement group action, followed by the Gowers-Cauchy-Schwarz inequality applied to the characteristic function.  This has the effect of taking iterated commutators on the group $G$ to reduce to the abelian case.  An analogous argument appears in the work of Green and Tao \cite{GT12} in the analysis of polynomial orbits on nilmanifolds.  At frequencies near the scale of the distribution, a Lindeberg replacement scheme is used to replace increments of the walk with those of a continuous compactly supported density with the same heat kernel, thus reducing to Alexopoulos' theorem.  It would be possible to make the replacement with increments of the heat kernel directly thus making the argument independent of \cite{A02a}, but the analysis then becomes technically more involved.   

 \subsection{Historical review} Central limit theorems on Lie groups have a long history, with early theorems proven by Wehn \cite{W62} and Tutubalin \cite{T64}, see also \cite{SV73}, \cite{CR78}, \cite{VSC92} and \cite{DS95}.  A central limit theorem with an optimal rate on stratified nilpotent groups is obtained under a homogeneous moment condition in \cite{P91}.

Alexopoulos, and Alexopoulos and Lohou\'{e} have made a detailed study of convolution powers of continuous densities, heat kernels and related questions on Lie groups, see \cite{A91}, \cite{A93}, \cite{A02a}, \cite{A02b}, \cite{A02c}, \cite{AL03}, and \cite{AL12}.  

A famous local limit theorem was proven by Bougerol \cite{B82} for convolutions of densities on a semi-simple group, which used the group's representation theory.  
There are still relatively few local limit theorems on non-abelian Lie groups that treat a measure which is not supported on a discrete group, or is not absolutely continuous with respect to Haar measure, of which \cite{B04}, \cite{B05} are an early example.  Recently,  Varj\'{u} has obtained such a local limit theorem for random walks on Euclidean space with transitions by a random isometry \cite{V15}.

\section*{Notation and conventions}
The connected, simply connected nilpotent Lie group of the theorem is $G$, with Lie algebra $\fg$ of dimension $q$.  The lower central series of $G$ is 
\begin{equation}
 G = G_1, \qquad G_{i+1} = [G, G_i], \; i \geq 1.
\end{equation}
The Lie algebra of $G_i$ is $\fg_i$.  The Lie algebra $\fg$ is identified with $\bR^q$ by choice of basis, which is fixed throughout the argument.  Vectors $x \in \bR^q$ are written in plain text and sequences of vectors $\ux \in (\bR^q)^N$ are written with an underscore.    The norm $\|\cdot\|$ refers to the $\ell^2$ norm on Euclidean space and is applied to $\fg$ and $\hat{\fg}$ by the fixed choice of basis.  The character on $\bR^q$ is written $e_\xi(x) = e^{2\pi i \xi \cdot x}$.  

A bump function $\sigma$ on $\bR^q$ is a non-negative $C^\infty$ function of compact support with integral 1.  Its dilation by $t > 0$ is indicated $\sigma_t(x) = t^q \sigma(tx)$.

Convolution of Borel probability measures $\mu, \nu$ on $G$ is defined weakly by, for $f \in C_c(G)$, 
\begin{equation}
 \langle f, \mu*\nu\rangle = \int_G \int_G f(gh) d\mu(g) d\nu(h).
\end{equation}
For $N \geq 1$, the iterated convolution $\mu^{*N}$ is defined by
\begin{equation}
 \mu^{*1} = \mu, \qquad \mu^{*(i+1)} = \mu * \mu^{*i},\; (i \geq 1).
\end{equation}
The characteristic function of a probability measure $\nu$ on $\bR^n$, respectively the Fourier transform of an $L^1$ function $f$, is defined to be
\begin{equation}
 \hat{\nu}(\xi) = \int_{\bR^n} e_{-\xi}(x)d\nu(x), \qquad \hat{f}(\xi) = \int_{\bR^n} f(x) e_{-\xi}(x)dx.
\end{equation}
If $f\in L^1$ is smooth, then the Fourier inversion integral is absolutely convergent, and
\begin{equation}
 f(x) = \int_{\bR^n} \hat{f}(\xi)e_{\xi}(x)d\xi.
\end{equation}

$C_2 = \zed/2\zed$ is the group of two elements.  For $\tau \in (C_2)^d$, $|\tau|$ is the Hamming or $\ell^1$ norm, which counts the number of non-zero coordinates.

 The notation $f = O_A(g)$ has the same meaning as $f \ll_A g$.  Both indicate that  $|f| \leq C(A) g$ for some constant $C(A)>0$ which depends at most on $A$ and possibly the structure of $G$.

\section{Nilpotent Lie groups}\label{nilpotent_group_section} A useful reference for the theory of nilpotent Lie groups is \cite{CG90}.

Given $G$, a connected, simply connected nilpotent Lie group with Lie algebra $\fg$ of dimension $q$, the exponential map $\exp$ is a diffeomorphism which identifies $G$ with $\fg$. Given a probability measure $\mu$ on $G$, denote $\mu_{\fg}$ its push-forward by the logarithm map to a probability measure on the Lie algebra; $\mu^{*k}_{\fg}$ should be read $(\mu^{*k})_{\fg}$ so that convolution is performed on $G$.  

Let the lower central series be defined by $\fg_1 = \fg$ and, for $i \geq 1$, $\fg_{i+1} = [\fg_1, \fg_i]$.  Since $\fg$ is nilpotent, one has the filtration
\begin{equation}
 \fg = \fg_1 \supset \fg_2 \supset \cdots \supset \fg_s \supset \fg_{s+1} = \{0\}, \qquad \fg_s \neq \{0\};
\end{equation}
$s$ is called the step of $\fg$. One can check, for $i + j \leq s+1$, $[\fg_i, \fg_j] \subset \fg_{i+j}$.  Also, one has $G_i = \exp \fg_i$ is the $i$th group in the lower central series of $G$. 

Let $\fa_1, \cdots, \fa_s$ be subspaces of $\fg$ such that, for each $1 \leq i \leq s$,
\begin{equation}
 \fg_i = \fa_i \oplus \cdots \oplus \fa_s.
\end{equation}
Let $d_i = \dim \fa_i$ and $q = \dim \fg$.  The homogeneous dimension of $G$ is 
\begin{equation}
 D = \sum_{i=1}^s i d_i.
\end{equation}
Assume given a basis $\{X_{i,j}\}_{\substack{1 \leq i \leq s\\ 1 \leq j \leq d_i}}$ of $\fg$ satisfying $\{X_{i,j}\}_{1 \leq j \leq d_i}$ is a basis for $\fa_i$. Identify $\fg$ with $\bR^q$ via 
\begin{equation}
 \psi: \begin{pmatrix} x^{(1,1)}\\ \vdots \\ x^{(s,d_s)} \end{pmatrix} = x^{(1,1)} X_{1,1}+ \cdots + x^{(s, d_s)} X_{s, d_s}.
\end{equation}
Having made this choice of coordinates, Haar measure on $G$ is normalized by pushing forward Lebesgue measure on $\fg$ by the exponential map. Exponential coordinates of the first kind are defined by
\begin{equation}
 \phi: \bR^q \to G, \qquad \phi: \begin{pmatrix} x^{(1,1)}\\ \vdots \\ x^{(s,d_s)} \end{pmatrix} \mapsto \exp \left( \psi(x)\right).
\end{equation}

Write for $1 \leq n \leq s$, 
\begin{equation}
 x^{(n)} = \begin{pmatrix} x^{(n,1)} \\ \vdots \\ x^{(n, d_n)}\end{pmatrix}, \qquad \xi^{(n)} = \begin{pmatrix} \xi^{(n,1)}\\ \vdots \\ \xi^{(n, d_n)} \end{pmatrix}
\end{equation}
for the coordinates at level $n$ in the filtration, respectively the corresponding dual frequencies in $\hat{\fg}$. These coordinates are said to have homogeneous degree $n$. For \begin{equation}S \subset \{(i,j): 1\leq i \leq s, 1 \leq j \leq d_i\},\end{equation} a monomial $m_\alpha(x) = \prod_{(i,j) \in S} \left(x^{(i,j)}\right)^{\alpha^{(i,j)}}$ with each $\alpha^{(i,j)} \neq 0$ has homogeneous degree 
\begin{equation}
\homdeg(m_\alpha) = \sum_{(i,j) \in S} i \alpha^{(i,j)}.
\end{equation}
The homogeneous degree of a monomial in several variables is defined to be the sum of the homogeneous degrees in the variables separately.  Note that the homogeneous degree is equal to the degree if and only if for every $(i,j) \in S$,  $i = 1$.
A probability measure $\mu$ on $G$ is said to have $d$ finite homogeneous moments if, for all monomials $m_\alpha$ of degree at most $d$,
\begin{equation}
 \int_{\fg} |m_\alpha(x)| d\mu_{\fg}(x)< \infty. 
\end{equation}

\subsection{Heat kernel and approximating continuous measure}Given a centered measure $\mu$ with two finite homogeneous moments on $G$, define the associated sub-Laplacian 
\begin{equation}
L_\mu = - \sum_{1 \leq i, j \leq d_1} a_{ij} X_{1,i} X_{1,j} - \sum_{ i \leq d_2} a_i X_{2,i}
\end{equation}
with coefficients
\begin{align}
 a_{ij} &= \frac{1}{2} \int x^{(1,i)}x^{(1,j)} d\mu_{\fg}(x),\\
\notag b_i &= \int x^{(2,i)} d\mu_{\fg}(x)
\end{align}
and 
\begin{equation}
 a_i = b_i - \frac{1}{2}\sum_{1 \leq j < k \leq d_1} a_{jk} \pr_{2,i}[X_{1,j}, X_{1,k}].
\end{equation}
Denote $u_t(x)$   the fundamental solution of the heat equation
\begin{equation}
\left(\frac{\partial}{\partial t} + L_\mu \right)u = 0, \qquad u_0 = \delta_{\id}.
\end{equation}

Fix $\varphi = \varphi(\mu)$  a continuous, compactly supported probability density on $G$ with push-forward $\varphi_{\fg}$ to $\fg$, which has first three homogeneous moments matching those of $\mu$, which are assumed finite.  
In particular, $\varphi$ is centered.
The existence of such a continuous $\varphi$ follows since $\mu_{\ab}$ has support generating a dense subgroup of $G_{\ab}$ and the quantities required on higher levels in the filtration are at most first degree in those variables.

Since the sub-Laplacian generated by $\mu$ and $\varphi$ depends only on the first two moments of the abelianized measures and the mean in $G_2/G_3$, the heat kernels of $\mu$ and $\varphi$ agree.  
By \cite{A02a} Theorem 1.9.1, Theorem \ref{main_theorem} holds with $\mu_\varphi = \varphi(g) dg$ in place of $\mu$. The argument presented reduces the local limit theorem for $\mu$ to that for $\mu_{\varphi}$.

\subsection{The product rule}
$G$ is identified with $(\bR^q, *)$ with the group law 
\begin{align}
 &\begin{pmatrix}
  x^{(1, 1)}\\ \vdots \\ x^{(s,d_s)}
 \end{pmatrix}*
 \begin{pmatrix}
  {y^{(1, 1)}} \\ \vdots \\ {y^{(s,d_s)}}
 \end{pmatrix}
 = \begin{pmatrix}
    {z^{(1, 1)}} \\ \vdots \\ {z^{(s,d_s)}}
   \end{pmatrix} = \phi^{-1}\left(\phi\begin{pmatrix}
  x^{(1, 1)}\\ \vdots \\ x^{(s,d_s)}
 \end{pmatrix} \cdot \phi \begin{pmatrix}
  {y^{(1, 1)}} \\ \vdots \\ {y^{(s,d_s)}}
 \end{pmatrix}\right) .
\end{align}
Given a sequence of vectors $\ux \in (\bR^q)^N$, write \begin{equation}\Pi(\ux) = x_1 * x_2 *\cdots *x_N \in \bR^q\end{equation} for their product.  The basic object of study is the characteristic function, for $\xi \in \hat{\fg} \cong \bR^q$,
\begin{equation}
 \chi_{N, \mu}(\xi) = \E_{\mu_{\fg}^{\otimes N}}\left[e_{\xi}(\Pi(\ux)) \right].
\end{equation}

In the case of a connected, simply connected nilpotent Lie group, the Baker-Campbell-Hausdorff formula is a finite expression that holds for all $X, Y \in \fg$,
\begin{equation}
 \log \left(\exp X \exp Y \right) = X + Y + \frac{1}{2} [X,Y] + \frac{1}{12}[X, [X,Y]] - \frac{1}{12}[Y, [X, Y]] + ....
\end{equation}
Using the Baker-Campbell-Hausdorff formula, the product rule for a sequence of group elements may be expressed as a polynomial map on the Lie algebra.  To describe this, given a sequence of elements $\ux = (x_k)_{k=1}^N$ of elements from $\bR^q$, sort a list of triples $\{(k_t, i_t, j_t)\}_{t = 1}^\ell$, $1 \leq k_t \leq N$, $1 \leq i_t \leq s$, $1 \leq j_t \leq d_{i_t}$ lexicographically. Say that the monomial 
\begin{equation}
 m_\alpha(\ux) =\prod_{t=1}^\ell \left(x_{k_t}^{(i_t, j_t)} \right)^{\alpha_{k_t}^{(i_t, j_t)}},
\end{equation}
 is \emph{initial} if it has the form 
\begin{equation}
 m_\alpha(\ux) = \prod_{k=1}^r \prod_{(i,j) \in S_k} \left(x_k^{(i,j)}\right)^{\alpha_k^{(i,j)}}
\end{equation}
where, for each $k$, $S_k$ is a non-empty subset of $\{(i,j): 1 \leq i \leq s, 1 \leq j \leq d_i\}$. If $m_\alpha$ is initial, say that monomial $m_\alpha'$ is of \emph{type} $m_\alpha$ if for some $\ell_1 < \ell_2 < \cdots < \ell_r$,
\begin{equation}
 m_\alpha'(\ux) = \prod_{k=1}^r \prod_{(i,j) \in S_k} \left(x_{\ell_k}^{(i,j)}\right)^{\alpha_k^{(i,j)}}.
\end{equation}

\begin{lemma}\label{product_lemma}
 Let $\ux = (x_1, ..., x_N)$ be a sequence of vectors from $\bR^q$ identified with coordinates on  the Lie algebra $\fg$. There are polynomials $\left\{P_N^{(i,j)}\right\}_{\substack{1 \leq i \leq s\\ 1 \leq j \leq d_i}}$ on $(\bR^q)^N$ satisfying the following conditions
 \begin{enumerate}
  \item (Degree bound) Each monomial $m_\alpha$ in $P_N^{(i,j)}$ satisfies $\homdeg(m_\alpha) \leq i$
  \item (Stability) If $m_\alpha$ appears in $P_N^{(i,j)}$ and if $M> N$ then $m_\alpha$ appears in $P_M^{(i,j)}$ with the same leading coefficient 
  \item (Invariance) If $m_\alpha'$ is of type $m_\alpha$, and if the maximum index of $m_\alpha'$ is at most $N$, then  $m_\alpha'$ appears in $P_N^{(i,j)}$ with the same leading coefficient as $m_\alpha$
 \end{enumerate}
such that the multiplication is given in coordinates by
\begin{equation}
\Pi(\ux) = \begin{pmatrix}\sum_{k=1}^N x_k^{(1,1)} + P_N^{(1,1)}(\ux)\\ \vdots \\ \sum_{k=1}^N x_k^{(s, d_s)} + P_N^{(s, d_s)}(\ux) \end{pmatrix}.
\end{equation}

\end{lemma}

\begin{proof}
 This follows from the Baker-Campbell-Hausdorff formula and induction.  Write $\ux'$ for the string $\ux$ with $x_N$ deleted, so that 
 \begin{equation}
  \Pi(\ux) = \Pi(\ux') * x_N.
 \end{equation}
 To obtain the degree bound, use that for $i + j \leq s+1$, $[\fg_i, \fg_j] \subset \fg_{i+j}$, so that, when taking commutators, the step in the filtration always increases at least as quickly as the homogeneous degree of the coefficient.  To obtain stability, note that if a contribution is made with a commutator involving $\psi(x_N)$ then the resulting monomial has an $x_N$ dependence, so that monomials without an $x_N$ dependence arise in $\Pi(\ux)$ only from the linear term in the Baker-Campbell-Hausdorff formula.  To prove invariance, let $m_\alpha'$ be a monomial appearing in $P_N$ which depends on $x_N$. Let the type of $m_{\alpha}'$ be $m_{\alpha}$.  Let $\tilde{\Pi}(\ux')$ be obtained from $\Pi(\ux')$ by setting to 0 all $x_j$ that do not appear in $m_{\alpha}'$ and write the remaining indices in order $\ell_1 < \ell_2 < \cdots < \ell_{r-1}$.  By induction, $\tilde{\Pi}(\ux') = \Pi(x_{\ell_1}, ..., x_{\ell_{r-1}})$ and thus the coefficients of $m_{\alpha}$ and $m_{\alpha}'$ are equal. 
\end{proof}

 Let 
 \begin{equation}
  m_\alpha = \prod_{k=1}^r \prod_{(i,j) \in S_k} \left(x_k^{(i,j)}\right)^{\alpha_k^{(i,j)}}
 \end{equation}
be an initial monomial of homogeneous degree $n$.  Given $\ux \in (\bR^q)^N$, define the generalized $U$-statistic
\begin{equation}
 U_\alpha(\ux) = \sum_{1 \leq \ell_1 < \ell_2 < \cdots < \ell_r } \prod_{k=1}^r \prod_{(i,j) \in S_k} \left(x_{\ell_k}^{(i,j)}\right)^{\alpha_k^{(i,j)}}.
\end{equation}
Lemma \ref{product_lemma} may be summarized as stating that
\begin{equation}
 \Pi^{(i,j)}(\ux) = \sum_{k = 1}^N x_k^{(i,j)} + P_N^{(i,j)}(\ux)
\end{equation}
where $P_N^{(i,j)}$ is a linear combination of generalized $U$-statistics of homogeneous degree at most $i$, with the $\ell^1$ norm of the coefficients in the linear combination bounded by a constant depending on the fixed choice of basis.
 
\subsection{Manipulations regarding the test function}\label{test_function_section}
The test function of the theorem takes the form, for $x \in \bR^q \cong \fg$,
\begin{align}
 L_gR_h f(\phi(x)) = f_{\fg}(\log g * x * \log h).
\end{align}
Applying the Baker-Campbell-Hausdorff formula, there are polynomials $p_{g,h}$ and $q_{g,h}$ satisfying for $1 \leq n \leq s$, $p_{g,h}^{(n)}, q_{g,h}^{(n)}$ are of homogeneous degree at most $n$, such that
\begin{equation}
x' = \log g * x * \log h = p_{g,h}(x), \qquad x = q_{g,h}(x'). 
\end{equation}
The relationship between $p_{g,h}^{(n)}$ and $q_{g,h}^{(n)}$ is linear in $x^{(n)}$ and ${x'}^{(n)}$ and polynomial in the lower degree coordinates.  In particular, $p$ can be obtained from $q$ by a polynomial change, and vice-versa, see \cite{GT12}, Appendix A for a further discussion.

Define the (naive) height $\Ht(p)$ of a polynomial $p$ to be the sup norm on its coefficients.  In particular, \begin{equation}\Ht(p_{g,h}) \ll 1+ \Ht(q_{g,h})^{O_s(1)}, \qquad \Ht(q_{g,h}) \ll 1 + \Ht(p_{g,h})^{O_s(1)}.\end{equation}

Let $\sigma \in C_c^\infty(\bR^q)$ be a smooth bump function with dilation, for $t > 0$, $\sigma_t(x) = t^q \sigma(tx)$.  Let $f_{\fg, t} = f_{\fg} * \sigma_t$ be the Euclidean convolution.  Since $f$ is assumed Lipschitz, \begin{equation}\left\|f_{\fg} - f_{\fg, t}\right\|_\infty= O\left( \frac{1}{t}\right)\end{equation} as $t \to \infty$.  Also, $\left\|f_{\fg, t}\right\|_1 \leq \left\|f_{\fg}\right\|_1$.
\begin{lemma}\label{Fourier_decay_lemma}
 For each $n \geq 1$ and for $\xi \in \bR^q$, $\|\xi\| \geq 1$ and $t > 1$, the Fourier transform
 \begin{equation}
  \widehat{L_gR_h f_{\fg, t}}(\xi) = \int_{\bR^q} f_{\fg, t}(\log g * x* \log h) e_{-\xi}(x) dx
 \end{equation}
satisfies
 \begin{equation}
  \left|\widehat{L_gR_h f_{\fg, t}}(\xi)\right| \leq O_{ n, D, f}(1)\left(1 + \Ht(p_{g,h}) \right)^{O_s(1)}  \left(\frac{t}{\|\xi\|}\right)^n \|f\|_1.
 \end{equation}

\end{lemma}

\begin{proof}
 Let $\xi_0 = \frac{\xi}{\|\xi\|}$ and integrate by parts $n$ times in the $\xi_0$ direction to obtain
 \begin{equation}
  \widehat{L_gR_h f_{\fg, t}}(\xi) = \left(\frac{1}{2\pi i \|\xi\|} \right)^n\int_{\bR^q} D_{\xi_0}^n\left[f_{\fg, t}(\log g * x *\log h)\right] e_{-\xi}(x) dx.
 \end{equation}
Write $\log g * x * \log h = p_{g,h}(x)$ and note that for $0 \leq j \leq n$,
\begin{equation}
 \left|D_{\xi_0}^j p_{g,h}(x)\right| \leq O_{n, D}(1)\Ht(p_{g,h})\|x\|^{s-j}
\end{equation}
By the compact support of $f_{\fg}$, restrict to $x' = p_{g,h}(x) \in \supp f$ which is $O_f(1)$.  Thus \begin{equation}\|x\| = q_{g,h}(x') = O_{D, f}(\Ht(q_{g,h})) = O_{D,f}\left((1 + \Ht(p_{g,h}))^{O_s(1)} \right). \end{equation}

Meanwhile $D_{\xi_0}^j f_{\fg, t} = f_{\fg} * D_{\xi_0}^j \sigma_t$, and $\left\|D_{\xi_0}^j \sigma_t\right\|_1 \ll_j t^j$.  Hence $\left\|D_{\xi_0}^j f_{\fg, t}\right\|_1 \ll_j t^j\|f\|_1$.  The conclusion now follows on applying the chain rule and bounding the integral in $L^1$. 
\end{proof}

The following lemma based on \cite{CW01} Theorem 2 is used to restrict the translations $g,h$ in Theorem \ref{main_theorem} to those for which $p_{g,h}$ has controlled height.
\begin{lemma}\label{small_values_lemma}
 Let $p: \bR^q \to \bR^q$ be a polynomial if degree at most $s$.  There is a constant $C = C(q, s)> 0$ such that, for any $\alpha > 0$,
 \begin{equation}
  \meas\left\{x \in \left[-\frac{1}{2}, \frac{1}{2}\right]^q: \|p(x)\| \leq \alpha\right\} \leq \frac{C \alpha^{\frac{1}{s}}}{\Ht(p)^{\frac{1}{s}}}.
\end{equation}
\end{lemma}
\begin{proof}
 The statement 
  \begin{equation}
  \meas\left\{x \in \left[-\frac{1}{2}, \frac{1}{2}\right]^q: \|p(x)\| \leq \alpha\right\} \leq \frac{C \alpha^{\frac{1}{s}}}{\left(\int_{\left[-\frac{1}{2}, \frac{1}{2}\right]^q} \|p\|^2\right)^{\frac{1}{2s}}}. 
 \end{equation}
 is a specialization of \cite{CW01} Theorem 2.  The conclusion follows since all norms on the space of degree $s$ polynomials on $\bR^q$ are equivalent.
\end{proof}

\section{Rearrangement group action}\label{action_section}
In \cite{DH15} Diaconis and the author used the following group action on strings. 
The group $C_2^{n-1}$ acts on strings of length $kn$ with the $j$th factor exchanging the relative order of the $j+1$st block of length $k$  with the previous $jk$ elements.  For instance, in the case $n = 4$, if $x_1, ..., x_4$ each represent a block of $k$ indices, the action is illustrated in
\begin{align}
 (1,0,0) \cdot \ux &= x_2 x_1 x_3x_4\\
 \notag (1,1,0) \cdot \ux &= x_3 x_2x_1x_4\\
 \notag (0,1,1) \cdot \ux &= x_4 x_3 x_1 x_2.
\end{align}
The relative order within the segments of length $k$ in each $x_i$ remains unchanged.

For $n \geq 2$, $k \geq 1$ and $1\leq N' \leq \left \lfloor \frac{N}{kn} \right \rfloor$ let \begin{equation}A_{k,n}^{N'} = \left(C_2^{n-1}\right)^{N'}\end{equation} act on strings of length $knN'$ with, for $j \geq 1$, the $j$th factor of $C_2^{n-1}$ in $A_{k,n}^{N'}$ acting as above on the contiguous subsequence of indices of length $k n$ ending at $j k n$.  The argument below considers $A_{k,n}^{N'}$ acting on substrings of length $k n N'$ within a string of length $N$.

\subsection{The Gowers-Cauchy-Schwarz inequality}
Given two elements \begin{equation}\utau_0, \utau_1 \in A_{k,n}^{N'}=\left(C_2^{n-1}\right)^{N'} = \left(C_2^{N'}\right)^{n-1}\end{equation} write $\utau = \left(\tau^{(1)}, \cdots, \tau^{(n-1)}\right)$. Thus $\tau_{0, i}^{(j)}$ is the element in $C_2$ which belongs to the $i$th factor of $(C_2^{n-1})$ in $A_{k,n}^{N'}$ and within this factor, the $j$th factor of $C_2$.  Given $s \in \{0,1\}^{n-1} \cong C_2^{n-1}$, define $\utau_s = \left(\tau_{s_1}^{(1)}, \cdots, \tau_{s_{n-1}}^{(n-1)} \right) \in A_{k,n}^{N'}$.

Since the increments of $\mu_{\fg}$ in the characteristic function $\chi_{N, \mu}$ are i.i.d., a further averaging may be introduced in which the group $A_{k,n}^{N'}$ acts on a substring of the product measure.  In general, let $P(\ux)$ be a continuous function of $\ux$ and let its characteristic function be
\begin{equation}
 \chi(\xi) = \E_{\mu_{\fg}^{\otimes N}}\left[e_\xi(P(\ux)) \right].
\end{equation}
Then
\begin{align}
 \chi(\xi) = \E_{\mu_{\fg}^{\otimes N}}\left[\E_{\utau \in A_{k,n}^{N'}}\left[e_{\xi}\left(P(\utau \cdot \ux) \right) \right] \right].
\end{align}
Writing $\utau = (\utau^{(1)}, \utau^{(2)}, ..., \utau^{(n-1)})$,
\begin{align}
 \chi(\xi) = \E_{\mu_{\fg}^{\otimes N}}\left[\E_{\utau^{(1)} \in C_2^{N'}} \cdots \E_{\utau^{(n-1)} \in C_2^{N'}}\left[e_{\xi}\left(P(\utau \cdot \ux) \right) \right] \right].
\end{align}
Moving one $\utau^{(i)}$ to the inside at a time and applying Cauchy-Schwarz to the inner expectation recovers the Gowers-Cauchy-Schwarz inequality:
\begin{align}
 \left|\chi(\xi)\right|^{2^{n-1}} \leq \E_{\mu_{\fg}^{\otimes N}}\left[\E_{\utau_0, \utau_1 \in A_{k,n}^{N'}} \left[e_\xi\left(\sum_{s \subset [n-1]} (-1)^{|s|} P(\utau_s \cdot \ux) \right) \right] \right].
\end{align}
In the case $\chi(\xi) = \chi_{N, \mu}(\xi)$, denote the right hand side $F\left(\xi, \mu; A_{k,n}^{N'}\right)$.

A basic lemma, which generalizes Lemma 24 of \cite{DH15}, is as follows.
\begin{lemma}\label{alternating_sum_lemma}
 Let $N, N' \geq 1$, let $k \geq 1$ and $n \geq 2$ be such that $k nN' \leq N$.  Let $\ux \in (\bR^q)^N$ and let $A_{k,n}^{N'}$ act on the substring of $\ux$ with indices in range, for some offset $o \geq 0$, $[o+1, o+knN']$. Define for $1 \leq i \leq nN'$, \begin{equation}\omega_i = \sum_{j=(i-1)k+1}^{ik}x_{o+j}.\end{equation}  For any $\utau_0, \utau_1 \in A_{k,n}^{N'}$ the summation formula holds,
 \begin{align}\label{summation_formula}
  \sum_{s \subset \{0,1\}^{n-1}}(-1)^{|s|} \Pi^{(n)}(\utau_s \cdot \ux) &= \sum_{i = 1}^{N'} \left(\sum_{s \in \{0,1\}^{n-1}} (-1)^{|s|}\Pi^{(n)}(\tau_{s, i}  \cdot (\omega_{n(i-1)+1}, ..., \omega_{ni}))  \right)
  \end{align}
  while for all $ n' < n$,
  \begin{align}\label{lower_phase_cancels} \sum_{s \subset \{0,1\}^{n-1}}(-1)^{|s|} \Pi^{(n')}(\utau_s \cdot \ux) &= 0.
 \end{align}
Moreover, 
\begin{align}\label{single_block_sum}&\sum_{s \in \{0,1\}^{n-1}} (-1)^{|s|}\Pi^{(n)}(\tau_{s, i}  \cdot (\omega_{1}, ..., \omega_{n}))\\&\notag = \left\{ \begin{array}{lll} (-1)^{|\tau_{0,i}|} \sum_{\tau \in C_2^{n-1}} (-1)^{|\tau|} \Pi^{(n)} \left(\tau \cdot  (\omega_{1}, ..., \omega_{n})\right) && \tau_{0, i} + \tau_{1, i} = (1)^{n-1}\\\\ 0 && \text{ otherwise}\end{array}\right..\end{align} 

In the $(1)^{n-1}$ case, the sum is a vector whose coordinates are non-zero multilinear polynomials in 
$ \left(\omega_1^{(1)}, \omega_2^{(1)}, \cdots, \omega_{n}^{(1)}\right)$.
\end{lemma}

\begin{proof}
 By the degree bound, all monomials appearing in $\Pi^{(n,j)}$ have homogeneous degree at most $n$, and hence degree at most $n$.  Given any collection of $m \leq n$ indices $k_1 < k_2 < \cdots < k_m$, if there is any bit $b \in C_2^{n-1}$ of Hamming weight 1 such that $\utau_s$ and $\utau_{s + b}$ act on $k_1 , ..., k_m$ leaving them in the same relative order, then by the invariance principle any monomials associated to these indices in the alternating sum 
\begin{equation}
 \sum_{s \subset \{0,1\}^{n-1}}(-1)^{|s|} \Pi^{(n, j)}(\utau_s \cdot \ux)
\end{equation}
cancel.  In particular, this occurs if $m < n$, or if the indices $k_1 < k_2 < \cdots < k_m$ are not acted on by the same factor of $C_2^{n-1}$ in $A_{k,n}^{N'}$, or if $k_1, k_2, ..., k_n$ do not appear in distinct blocks in the action, or if the corresponding factor of $\utau_0$ and $\utau_1$ do not add to the all 1's element. In particular this proves (\ref{lower_phase_cancels}).

Since the only surviving monomials have degree $n$ and homogeneous degree $n$, the surviving variables are all from the first level of the filtration $x^{(1)}$ and all of the monomials are linear in each variable.  By the invariance principle, the surviving polynomial is in fact a polynomial on the sums $\omega_1, \omega_2, ..., \omega_n$.  

The formula (\ref{single_block_sum}) is immediate.  To prove that the $(1)^{n-1}$ case of (\ref{single_block_sum}) is non-vanishing, in the case $k = 1$ let $g_i = \exp(\omega_i) \in G$.  The sum
\begin{equation}\label{poly_formula}
  \sum_{\tau \in C_2^{n-1}} (-1)^{|\tau|} \Pi^{(n,j)} \left(\tau \cdot  (\omega_{1}, ..., \omega_{n})\right) 
\end{equation}
is equal to the $X^{(n,j)}$ coordinate in the logarithm of the iterated commutator
\begin{equation}
\label{commutator_formula} \left[     \cdots    \left[\left[\left[ g_1, g_2\right], g_3\right], g_4\right], \cdots, g_{n} \right].
\end{equation}
To verify this by induction, note that the commutator may be calculated in $G_n/G_{n+1}$, which is abelian, and depends only on $g_{n}$ in $G_1/G_2$, so that the calculation may be performed using first commutators in the Lie algebra.

Since commutators of the type (\ref{commutator_formula}) generate $G_n$, it follows that (\ref{poly_formula}) is non-zero.
 \end{proof}
 
Given probability measure $\mu$ on $G$, let $\mu_{n}$ be the probability measure on $G_n/G_{n+1}$ with distribution 
 \begin{equation}
  \left[\cdots  \left[ \left[g_1, g_2\right], g_3 \right], \cdots, g_{n}\right], \qquad g_i \text{ i.i.d. } \mu.
 \end{equation}
Thus $\mu_1 = \mu_{\ab}$, and for $n \geq 2$, $\mu_{n}$ has distribution given by  
\begin{equation}\sum_{\tau \in C_2^{n-1}} (-1)^{|\tau|} \Pi^{(n)} \left(\tau \cdot  (\omega_{1}, ..., \omega_{n})\right)\end{equation} in which the $\omega_i$ are drawn i.i.d. from  $\mu_{\fg}$.
Given $\xi^{(n)} \in \widehat{G_n /G_{n+1}} \cong \widehat{\fg_n/\fg_{n+1}}$ denote the characteristic function of $\mu_n$ by
\begin{equation}
 F_{n, \mu}\left(\xi^{(n)}\right) = \hat{\mu}_n\left(\xi^{(n)}\right).
\end{equation}
\begin{lemma}\label{gcs_lemma}
 Let $2 \leq n \leq s$, $k, N' \geq 1$, and let $N \geq knN'$.  Let $A_{k,n}^{N'}$ act on a substring of $\ux \in (\bR^q)^N$ as above.  Let $\xi \in \hat{\fg}$ satisfy $\xi^{(j)} = 0$ for all $j > n$.  Then
 \begin{equation}
  F\left(\xi, \mu; A_{k,n}^{N'}\right) = \left(1 - \frac{1}{2^{n-1}} + \frac{\RE \left[F_{n, \mu^{*k}}\left(\xi^{(n)}\right)\right]}{2^{n-1}} \right)^{N'}.
 \end{equation}

\end{lemma}

\begin{proof}
 The expectation factors through the product structure of $A_{k,n}^{N'}$ due to the summation formula (\ref{summation_formula}).  In the expectation over $A_{k,n}^{N'}$ the probability that $\tau_{0, j} + \tau_{1, j} = (1)^{n-1}$ is $\frac{1}{2^{n-1}}$, and conditioned on this, the expectation over the corresponding block is $(-1)^{|\tau_{0,i}|} F_{n,\mu^{*k}}\left(\xi^{(n)}\right)$.  The real part occurs since conditionally, $|\tau_{0,i}|$ has parity 0 and 1 with equal probability.
\end{proof}

\subsection{The Cram\'{e}r condition}
A probability measure $\nu$ on $\bR^m$ has characteristic function $\hat{\nu}$ satisfying the Cram\'{e}r condition if there exists $0 < \epsilon < 1$ such that
\begin{equation}
 \sup_{\xi \in \bR^m, \|\xi\|> 1} |\hat{\nu}(\xi)| \leq 1-\epsilon.
\end{equation}
The condition is equivalent to the statement, for all $r >0$ there exists $0 < \epsilon(r) < 1$ such that
\begin{equation}
 \sup_{\xi \in \bR^m, \|\xi\| > r} |\hat{\nu}(\xi)| \leq 1-\epsilon(r).
\end{equation}
The equivalence may be checked by noting $1 - |\hat{\nu}(\xi_1 + \xi_2)| \leq 2(2 - |\hat{\nu}(\xi_1)| - |\hat{\nu}(\xi_2)|)$, see \cite{TV06} p. 183, where the proof does not use that the probability measure has finite support.

\begin{lemma}\label{cramer_lemma}
 Let $\mu$ be a centered probability measure on $G$ satisfying $G = \overline{\left \langle \supp \mu \right \rangle}$, whose abelianization $\mu_{\ab}$ has characteristic function satisfying the Cram\'{e}r condition.  For each $2 \leq n \leq s$ the measure $\mu_{n}$ on  $G_n/G_{n+1}$  has characteristic function satisfying the Cram\'{e}r condition.
\end{lemma}

\begin{proof}
 Write 
 \begin{equation}
   \left[\cdots  \left[ \left[g_1, g_2\right], g_3 \right], \cdots, g_{n}\right] \bmod G_{n+1} = \left\langle \lambda(g_1,g_2, ..., g_{n-1}), g_{n}\right\rangle
 \end{equation}
 in which $\lambda(g_1,g_2, ..., g_{n-1})$ is a linear map from $G_{\ab}$ to $G_n/G_{n+1}$.  Recall that $\lambda$ itself is multilinear in $g_1, g_2, ..., g_{n-1} \bmod G_2$.  Given $\xi \in \widehat{G_n/G_{n+1}}$, $\|\xi\| \geq 1$ one has \begin{equation}\xi\left( \left[\cdots  \left[ \left[g_1, g_2\right], g_3 \right], \cdots, g_{n}\right]\right) = \left(\xi \cdot \lambda(g_1, g_2, ..., g_{n-1}) \right) \left(g_{n} \right). \end{equation}
 Since the semigroup generated by $\supp \mu_{\ab}$ is dense in $G_{\ab}$, and since $G_n = [G_{n-1}, G_1]$ is equal to the set of commutators of the stated type on $G$, and since $\lambda$ is multilinear, it follows that $\xi$ does not annihilate $\lambda(g_1, ..., g_{n-1})$ with positive probability, and hence for some $r > 0$, $\left\|\xi \cdot \lambda(g_1,g_2, ..., g_{n-1})\right\| > r$ with positive probability.  Integrating, this suffices to obtain the Cram\'{e}r condition.
\end{proof}

\begin{lemma}\label{char_fun_bound_lemma}
 There is a constant $c = c(\mu)>0$ such that, for each $1 \leq n \leq s$, for all $\xi^{(n)} \neq 0$, when $k$ is assigned by the rule
 \begin{equation}
  k = \left\{\begin{array}{ccc}\left\lfloor \frac{1}{\left\|\xi^{(n)}\right\|^{\frac{2}{n}}}\right\rfloor, && \left\|\xi^{(n)}\right\| \leq 1\\ 1, && \left\|\xi^{(n)}\right\| >1 \end{array}\right.,
 \end{equation}
 one has
$
 \left|F_{n, \mu^{*k}}\left(\xi^{(n)}\right) \right| \leq 1- c.$

\end{lemma}

\begin{proof}
 For any fixed $r>0$, for $\left\|\xi^{(n)}\right\| \geq r$ this follows from the Cram\'{e}r condition.  Otherwise, using the description (\ref{single_block_sum}), it follows from the functional central limit theorem that when $\ux$ is drawn from $(\mu^{*k})^{\otimes n}$, 
  \begin{equation}\frac{1}{k^{\frac{n}{2}}} \left(\sum_{\tau \in C_2^{n-1}} (-1)^{|\tau|} \tau \cdot \Pi^{(n)}(\ux) \right)\end{equation}
converges to a continuous probability density.  Since $\left\|\xi^{(n)}\right\| \asymp \frac{1}{k^{\frac{n}{2}}}$, the claim follows.

\end{proof}

\section{Estimates of moments}
Throughout this section $\mu$ is a centered probability measure on $G$.
\begin{lemma}\label{homogeneous_poly_lemma}
Let $m, n, N \geq 1$ and suppose that $\mu$ has $2mn$ finite homogeneous moments.  For all generalized $U$-statistics $U_\alpha$ of homogeneous degree $n$, 
\begin{equation}
 \E_{\mu_\fg^{\otimes N}}\left[\left|U_\alpha(\ux)\right|^{2m} \right] \leq O_{\mu, m n}(1)N^{mn}.
\end{equation}
\end{lemma}

\begin{proof}

Let $a$ be the number of indices in $m_\alpha$ of homogeneous degree 1, and note that $n \geq a + 2(r-a) = 2r-a$.  On expanding $\left|U_\alpha(\ux)\right|^{2m}$ and performing expectation, any monomials that have indices which appear with homogeneous degree 1 have expectation 0.  Those remaining monomials have homogeneous degree at least 2 in every coordinate upon which the expectation depends, and hence have expectation $O_{\mu, mn}(1)$ by the moment condition.  The total number of indices which may appear in such a monomial is at most $(2m)(r-a) + am \leq nm$.  Counting the number of monomials with non-vanishing expection by letting $L$ be the number of indices appearing, the  expectation is bounded by
\begin{equation}
 \E\left[\left|U_\alpha(\ux)\right|^{2m}\right] \leq O_{\mu, mn}(1) \sum_{L = 1}^{nm} \binom{N}{L}L^{2rm} \leq  O_{\mu, mn}(1) N^{nm}.
\end{equation}

\end{proof} 
 
\begin{lemma}\label{truncation_lemma}
 For each $A, \delta>0$ there is $C(A, \delta)>0$ such that, if $\mu$ has $C(A, \delta)$ finite homogeneous moments then
 \begin{equation}
  \mu_{\fg}^{\otimes N}\left\{ \max_{n} \frac{1}{N^{\frac{n}{2}}}\left\|\Pi^{(n)}(\ux)\right\| > N^\delta\right\} = O_{\mu, A, \delta}\left(N^{-A}\right).
 \end{equation}

\end{lemma}

\begin{proof}
 If $\mu$ has $2mn$ homogeneous moments then the estimate
 \begin{equation}
  \E_{\mu_{\fg}^{\otimes N}}\left[\left\| \Pi^{(n)}(\ux)\right\|^{2m} \right] \leq O_{\mu, mn}(1)N^{mn}
 \end{equation}
 follows by repeatedly applying the power mean inequality to first estimate $\E\left[\left\|\Pi^{(n)}\right\|^{2m}\right]$ in terms of moments of the individual coordinates $\E\left[\left\|\Pi^{(n,j)}(\ux)\right\|^{2m}\right]$ and then in terms of the moments of individual $U$ statistics of homogeneous degree at most $n$, to which Lemma \ref{homogeneous_poly_lemma} applies. 
 
 The claim now follows by taking a high enough moment and applying Markov's inequality.

\end{proof}

\begin{lemma}\label{tail_lemma}
Let $m \geq 1$ and $1 \leq n \leq s$, and assume that $\mu$ has $2mn$ fininte homogeneous moments.  For all $N' \leq N$, when $\ux = \ux_0 \oplus \ux_t$ is the concatenation of strings of length $N'$ and $N-N'$
 \begin{equation}
  \E_{\mu_\fg^{\otimes N}}\left[\left\|\Pi^{(n)}(\ux) - \Pi^{(n)}(\ux_t)\right\|^{2m} \right] \leq O_{\mu, mn}(1) N^{mn}\left( \frac{N'}{N}\right)^m.
 \end{equation}
\end{lemma}

\begin{proof}
By repeatedly applying the power mean inequality it suffices to prove, for any generalized $U$-statistic  
\begin{equation}
U_\alpha(\ux) = \sum_{1 \leq \ell_1 < \ell_2 < \cdots < \ell_r } \prod_{k=1}^r \prod_{(i,j) \in S_k} \left(x_{\ell_k}^{(i,j)}\right)^{\alpha_k^{(i,j)}}
\end{equation}
of homogeneous degree at most $n$, 
the estimate
 \begin{equation}
  \E_{\mu_\fg^{\otimes N}}\left[\left|U_\alpha(\ux) - U_\alpha(\ux_t)\right|^{2m} \right] \leq O_{\mu, mn}(1) N^{mn}\left( \frac{N'}{N}\right)^m.
 \end{equation}
Define polynomials, for $1 \leq a \leq r$, $U_\alpha^{0, a}$, $U_\alpha^{t, a}$,
\begin{align}
 U_{\alpha}^{0,a}(\ux) &= \sum_{1 \leq \ell_1 < \ell_2 < \cdots < \ell_a } \prod_{k=1}^a \prod_{(i,j) \in S_k} \left(x_{\ell_k}^{(i,j)}\right)^{\alpha_k^{(i,j)}} \\
 \notag U_\alpha^{t, a}(\ux) &= \sum_{1 \leq \ell_1 < \ell_2 < \cdots < \ell_a } \prod_{k=1}^a \prod_{(i,j) \in S_{r-a + k}} \left(x_{\ell_k}^{(i,j)}\right)^{\alpha_k^{(i,j)}},
\end{align}
and also, make the convention that $U_\alpha^{0,0} = U_\alpha^{t, 0} = 1$.  Hence,
\begin{equation}
 U_\alpha(\ux) - U_{\alpha}(\ux_t) = \sum_{a = 1}^r U_\alpha^{0,a}(\ux_0) U_\alpha^{t, r-a}(\ux_t).
\end{equation}
Applying the power mean inequality one further time, it suffices to prove the estimate, for each $1 \leq a \leq r$, 
\begin{equation}
 \E_{\mu_\fg^{\otimes N}}\left[\left|U_\alpha^{0,a}(\ux_0) U_\alpha^{t, r-a}(\ux_t) \right|^{2m} \right] \leq O_{\mu, mn}(1) N^{mn}\left(\frac{N'}{N} \right)^m.
\end{equation}
Since
\begin{align}
 \E_{\mu_\fg^{\otimes N}}\left[\left|U_\alpha^{0,a}(\ux_0) U_\alpha^{t, r-a}(\ux_t) \right|^{2m} \right] &= \E_{\mu_{\fg}^{\otimes N'}}\left[\left|U_\alpha^{0,a}(\ux_0)\right|^{2m} \right]\E_{\mu_{\fg}^{\otimes (N-N')}}\left[\left| U_\alpha^{t, r-a}(\ux_t)\right|^{2m} \right]
\end{align}
the claim follows from Lemma \ref{homogeneous_poly_lemma}, since, for each $1 \leq a \leq r$, \begin{equation}\homdeg(U_\alpha^{0, a}) + \homdeg(U_\alpha^{t, r-a}) \leq n, \qquad \homdeg(U_\alpha^{t, r-a}) \leq n-1.\end{equation}

\end{proof}

 Denote $\Pi_j(\ux)$ the part of $\Pi(\ux)$ which depends on $x_j$.  Set $\Pi_j^{\leq 3}(\ux)$ (resp. $\Pi_j^{> 3}(\ux)$) the part of $\Pi_j(\ux)$ which is of homogeneous degree $\leq 3$ (resp. $>3$) in $x_j$, and for $d= 1, 2, 3$, $\Pi_j^d(\ux)$ the part of $\Pi_j(\ux)$ which is of homogeneous degree $d$ in $x_j$. $\Pi_j^{*, (n)}$ denotes the part of $\Pi_j^*$ at level $n$.  Use the same notation with $\Pi$ replaced with a $U$-statistic $U_\alpha$.
 
 \begin{lemma}\label{coordinate_bounds_lemma}
 Assume that $\mu$ has at least $6s$ homogeneous moments.
 Let $N \geq 1$,  $1 \leq j \leq N$ and $\umu = \mu^{\otimes (j-1)} \otimes \mu_{\varphi}^{\otimes (N-j + 1)}$ or $\umu = \mu^{\otimes j} \otimes \mu_{\varphi}^{\otimes (N-j)}$.  For $k = 1, 2, 3$, and $m \leq 3$,
 \begin{equation}\label{low_degree_bound}
  \E_{\umu}\left[\left|\xi \cdot \Pi_j^k(\ux)\right|^{2m} \right] = O_\mu\left(\sum_{\ell = k}^s \|\xi_\ell\|^{2m} N^{(\ell-k)m} \right)
 \end{equation}
and
\begin{equation}\label{higher_degree_bound}
 \E_{\umu}\left[\left|\xi \cdot \Pi_j^{>3}(\ux)\right|^2 \right] = O_\mu\left(\sum_{\ell = 4}^{s} \|\xi_\ell\|^2 N^{\ell-4} \right).
\end{equation}

 \end{lemma}

\begin{proof}
 By the power mean inequality, then Cauchy-Schwarz,
 \begin{align}
  \E_{\umu}\left[\left|\xi \cdot \Pi_j^k(\ux)\right|^{2m} \right] &\ll_s \sum_{\ell = k}^s  \E_{\umu}\left[\left|\xi^{(\ell)} \cdot \Pi_j^{k, (\ell)}(\ux)  \right|^{2m} \right]\\
  \notag & \leq \sum_{\ell = k}^s \left\|\xi^{(\ell)}\right\|^{2m} \E_{\umu}\left[\left\|\Pi_j^{k, (\ell)}(\ux) \right\|^{2m} \right].
 \end{align}
Applying the power mean inequality several further times to first replace $\Pi_j^{k, (\ell)}$ with its individual coordinates, then with an individual generalized $U$-statistic reduces to proving the bound for a $U$-statistic $U_\alpha$ of homogeneous degree $\ell \geq k$,
\begin{equation}
 \E_{\umu}\left[ \left|U_{\alpha, j}^k(\ux)\right|^{2m} \right] = O_\mu\left(N^{(\ell-k)m}\right).
\end{equation}
Let $\ux_0$, $\ux_t$ denote the substrings of $\ux$ prior to $j$ and after $j$ respectively.  The claim follows from Lemma \ref{homogeneous_poly_lemma} after factoring  \begin{equation} U_{\alpha, j}^k(\ux) = U_1(\ux_0) m(x_j)U_2(\ux_t)\end{equation} where $m$ is a monomial of homogeneous degree $k$ and $U_1$ and $U_2$ are $U$-statistics satisfying $\homdeg(U_1) + \homdeg(U_2) = \ell-k$.

The proof of (\ref{higher_degree_bound}) is similar.
\end{proof}

\section{Proof of Theorem \ref{main_theorem}}
The following lemma is used to truncate in frequency space to the scale of the distribution.

\begin{lemma}\label{frequency_truncation_lemma}
 Let $N \geq 1$, $A>0$ and let $1 > \epsilon_1 > \epsilon_2 > \cdots > \epsilon_s > \epsilon_{s+1}=0$ be a collection of constants satisfying for all $1 \leq n < s$, $\epsilon_n > n \epsilon_{n+1}$. Suppose that $\mu$ has $C(A, \epsilon)$ finite homogeneous moments for some constant  $C(A, \epsilon)>0$. If
 \begin{equation}
  \max\left\{ \left\|\xi^{(n)}\right\| N^{\frac{n}{2} - \epsilon_n}: 1\leq n \leq s \right\} > 1,
 \end{equation}
 then $
 |\chi_{N, \mu}(\xi)| = O_{\mu, A, \epsilon}\left(N^{-A}\right).$

\end{lemma}

\begin{proof}
 Let $n$ be maximal such that $\left\|\xi^{(n)}\right\| N^{\frac{n}{2}-\epsilon_n} > 1$.  If $n = s$, set $N' = N$, otherwise, set $N' = \left \lfloor N^{1-\frac{\epsilon_n}{n} - \epsilon_{n+1}} \right \rfloor$. Let $\ux_0$ and $\ux_t$ be strings of vectors from $\bR^q$ of lengths $N'$ and $N-N'$ and let $\ux = \ux_0 \oplus \ux_t$ be the concatenation.    Let 
 \begin{equation}
  \Xi_{n+1}(\ux) = \sum_{j = n+1}^s \xi^{(j)} \cdot \Pi^{(j)}(\ux).
 \end{equation}
 Denote $T_{2m-1}(x) = \sum_{j=0}^{2m-1} \frac{(2\pi i x)^j}{j!}$ the degree $2m-1$ Taylor expansion of $e^{2\pi i x}$ and recall that Taylor's theorem with remainder gives
\begin{equation}
 \left|T_{2m-1}(x) - e^{2\pi i x}\right| \leq \frac{(2\pi x)^{2m}}{(2m)!}.
\end{equation}
It follows that
\begin{align}\label{chi_expectation}
 \chi_{N, \mu}(\xi)= & \E_{\mu_{\fg}^{\otimes N}}\left[ \prod_{j=1}^n e_{\xi^{(j)}}(\Pi^{(j)}(\ux)) \prod_{\ell=n+1}^s e_{\xi^{(\ell)}}(\Pi^{(\ell)}(\ux_t)) T_{2m-1}\left(\Xi_{n+1}(\ux) - \Xi_{n+1}(\ux_t) \right) \right]
 \\\notag & + O_m\left( \E_{\mu_{\fg}^{\otimes N}}\left[\left|\Xi_{n+1}(\ux) - \Xi_{n+1}(\ux_t) \right|^{2m} \right]\right).
\end{align}
 
By Lemma \ref{tail_lemma} and H\"{o}lder's inequality, if $\mu$ has sufficiently many homogeneous moments, 
\begin{align}
 \E_{\mu_{\fg}^{\otimes N}}\left[\left|\Xi_{n+1}(\ux) - \Xi_{n+1}(\ux_t) \right|^{2m} \right] &\ll_{\mu, m} \sum_{j = n+1}^s \left\|\xi^{(j)}\right\|_{2m}^{2m} N^{m(j-1)} (N')^m\\
 \notag & \leq \sum_{j=n+1}^s \left(N^{-\frac{j}{2} +\epsilon_j} \right)^{2m} N^{m(j-1)} (N')^m\\
 \notag & \leq \sum_{j=n+1}^s N^{m \left(2 \epsilon_j - \frac{\epsilon_n}{n} -\epsilon_{n+1} \right)}.
\end{align}
Since each exponent is negative, the sum may be made $O_{\mu, A, \epsilon}\left( N^{-A}\right)$ by choosing $m$ sufficiently large in terms of $A$ and $\epsilon$.

Expand $T_{2m-1}\left(\Xi_{n+1}(\ux) - \Xi_{n+1}(\ux_t) \right)$ into monomials of degree bounded by $(2m-1)s$ with coefficients of $\ell^1$ norm bounded by $\ll_B N^B$.  Set $N_1 = \left\lfloor \frac{N'}{2ms} \right \rfloor$.  Given a typical monomial $M$, let $[J+1, J+N_1]$ be a set of indices from $[1, N']$ which does not meet $M$. 

Define, 
\begin{equation}
 k = \left\{\begin{array}{ccc}\left\lfloor \frac{1}{\left\|\xi^{(n)}\right\|^{\frac{2}{n}}}\right\rfloor, && \|\xi^{(n)}\| \leq 1\\ 1, && \|\xi^{(n)}\| >1 \end{array}\right., \qquad N_1'= \left \lfloor \frac{N_1}{kn} \right \rfloor.
\end{equation}
  Let $A_{k, n}^{N_1'} = (C_2^{n-1})^{N_1'}$ act on the substring $[J+1, J+N_1]$ as described in Section \ref{action_section}. Since the monomial $M$ is invariant under the group action, its contribution to the expectation (\ref{chi_expectation}) is given by
\begin{align}\label{action_expectation}
 \E_{\ux_t \sim \mu_{\fg}^{\otimes (N-N')}}\left[M\prod_{j=n+1}^s e_{\xi^{(j)}}(\Pi^{(j)}(\ux_t)) \E_{\ux_0 \sim \mu_{\fg}^{\otimes N'}} \left[\E_{\utau \in A_{k,n}} \left[ \prod_{j=1}^n e_{\xi^{(j)}}(\Pi^{(j)}(\utau \cdot\ux)) \right]\right] \right].
\end{align}
By Cauchy-Schwarz,
\begin{align}\notag 
 |(\ref{action_expectation})|^2 \leq &\E_{\ux_t \sim \mu_{\fg}^{\otimes (N-N')}}\left[|M|^2 \right]\\& \notag \times \E_{\ux_t \sim \mu_{\fg}^{\otimes (N-N')}}\left[\left|  \E_{\ux_0 \sim \mu_{\fg}^{\otimes N'}} \left[\E_{\utau \in A_{k,n}} \left[ \prod_{j=1}^n e_{\xi^{(j)}}(\Pi^{(j)}(\utau \cdot\ux))) \right]\right]\right|^2 \right].
\end{align}
Bound the first expectation by a constant.  In the case $n \geq 2$,  apply Gowers-Cauchy-Schwarz to bound the second expectation, using Lemma \ref{gcs_lemma} to evaluate the expectation that results.  In either the case $n=1$ or $n\geq 2$, it follows from Lemma \ref{char_fun_bound_lemma} that 
\begin{align}
 |(\ref{action_expectation})|^{2^{n-1}} \ll_\mu \left(1 - \frac{1}{2^{n-1}} + \frac{\left|F_{n, \mu_{\fg}^{*k}}(\xi^{(n)})\right|}{2^{n-1}} \right)^{N_1'} \ll_\mu \exp\left(-CN_1' \right).
\end{align}
 Since 
\begin{equation}
 N_1' \gg \frac{N_1}{k} \gg \min\left(1, \left\|\xi^{(n)}\right\|^{\frac{2}{n}}\right) N^{1-\frac{\epsilon_n}{n} -\epsilon_{n+1}} \gg N^{\frac{\epsilon_n}{n} - \epsilon_{n+1}} 
\end{equation}
tends to infinity with $N$ like a small power of $N$, the exponential savings dominates the polynomial bound on the coefficients of the monomials, which proves the lemma.
\end{proof}

 The following Lindeberg exchange lemma approximates the distribution of $\mu^{*N}$ with that of $\mu_{\varphi}^{*N}$.
\begin{lemma}\label{Lindeberg_lemma}
 Let $N \geq 1$. Let $1 > \epsilon_1 > \epsilon_2 > \cdots > \epsilon_s > \epsilon_{s+1}=0$ be a collection of constants satisfying for all $1 \leq n < s$, $\epsilon_n > n \epsilon_{n+1}$. Assume that \begin{equation}
  \max\left\{ \left\|\xi^{(n)}\right\| N^{\frac{n}{2} - \epsilon_n}: 1\leq n \leq s \right\} \leq 1.
 \end{equation}
 Then 
 \begin{equation}
  \left|\chi_{N, \mu}(\xi) - \chi_{N, \mu_{\varphi}}(\xi) \right| = O_\mu\left( N^{-1 + O(\epsilon_1)}\right).
 \end{equation}

\end{lemma}

\begin{proof}

 Define
 \begin{equation}
  \Delta_j =  \E_{\mu_{\fg}^{\otimes j} \otimes \mu_{\varphi, \fg}^{\otimes (N-j)}}\left[ e_\xi(\Pi(\ux))\right] -\E_{\mu_{\fg}^{\otimes (j-1)} \otimes \mu_{\varphi, \fg}^{\otimes (N-j+1)}}\left[ e_\xi(\Pi(\ux))\right] 
 \end{equation}
so that, by the triangle inequality,
\begin{equation}
 \left|\chi_{N, \mu}(\xi) - \chi_{N, \mu_{\varphi}}(\xi)\right| \leq \sum_{j=1}^N |\Delta_j|.
\end{equation}

 For $1 \leq j \leq N$ bound,  moving expectation against $x_j$ to the inside and using the triangle inequality, 
 \begin{align}
 |\Delta_j| \leq \E_{\mu_{\fg}^{\otimes (j-1)}\otimes \mu_{\varphi, \fg}^{\otimes (N-j)}}\left[\left| \int_{\fg} e_{\xi}(\Pi_j(\ux))d\mu_{\fg}(x_j) - \int_{\fg}e_{\xi}(\Pi_j(\ux))   \varphi_{\fg}(x_j)dx_j \right|  \right].
 \end{align}
Using $|e(x) - e(y)| \leq 2\pi |x-y|$ and the triangle inequality, the right hand side is bounded by a constant times
\begin{align}
 &\E_{\mu_{\fg}^{\otimes j} \otimes \mu_{\phi, \fg}^{\otimes (N-j)}}\left[\left|\xi \cdot \Pi_j^{>3}(\ux) \right|\right]+\E_{\mu_{\fg}^{\otimes (j-1)} \otimes \mu_{\varphi, \fg}^{\otimes (N-j+1)}}\left[\left|\xi \cdot \Pi_j^{>3}(\ux) \right|\right]\\
 &\notag +\E_{\mu_{\fg}^{\otimes (j-1)}\otimes \mu_{\varphi, \fg}^{\otimes (N-j)}}\left[\left| \int_{\fg} e_{\xi}\left(\Pi_j^{\leq 3}(\ux)\right)d\mu_{\fg}(x_j)  - \int_{\fg}e_{\xi}\left(\Pi_j^{\leq 3}(\ux)\right)   \varphi_{\fg}(x_j)dx_j\right|  \right].
\end{align}
Note that $\Pi_j^{(n), > 3}$ is of homogeneous degree $\leq n-4$ in the variables other than $x_j$, while $\|\xi^{(n)}\| \leq N^{-\frac{n}{2} + \epsilon_n}$.  Thus, by Cauchy-Schwarz and (\ref{higher_degree_bound}) of Lemma \ref{coordinate_bounds_lemma}, the top line is $O_\mu\left(N^{-2 + O(\epsilon_1)}\right)$.

By Taylor expansion,
\begin{align}
  e_{\xi}\left(\Pi_j^{\leq 3}(\ux)\right)& = 1 + i 2\pi  \xi \cdot \Pi_j^1(\ux) - \frac{1}{2} \left( 2\pi  \xi \cdot \Pi_j^1(\ux) \right)^2 - \frac{i}{6} \left( 2\pi  \xi \cdot \Pi_j^1(\ux) \right)^3 \\ \notag &+ i2 \pi \xi \cdot \Pi_j^2(\ux) - 4\pi^2 (\xi \cdot \Pi_j^1(\ux))(\xi \cdot \Pi_j^2(\ux)) + i 2\pi \xi \cdot \Pi_j^3(\ux) \\
  & \notag + O\left(\left|\xi \cdot \Pi_j^1(\ux) \right|^4 + \left(1 + |\xi \cdot \Pi_j^1(\ux)|^3 \right)\left(|\xi\cdot \Pi_j^2(\ux)|^2 + |\xi \cdot \Pi_j^3(\ux)|^2 \right) \right).
\end{align}
Since the main term has homogeneous degree at most 3 in $x_j$, and since the first three homogeneous moments of $\mu$ and $\mu_{\varphi}$ agree, the integral of these terms cancel. In the error term, separate 
$|\xi \cdot \Pi_j^1(\ux)|^3$ from $|\xi\cdot \Pi_j^2(\ux)|^2$ and $|\xi \cdot \Pi_j^3(\ux)|^2$ with Cauchy-Schwarz.  Now applying (\ref{low_degree_bound}) of Lemma \ref{coordinate_bounds_lemma}, the error term is bounded in expectation by $O_\mu\left(N^{-2 + O(\epsilon_1)}\right)$ as before.

\end{proof}

\begin{proof}[Proof of Theorem \ref{main_theorem}]
Let $f$ be the Lipschitz, compactly supported test function of the theorem, let $f_{\fg}$ be the push-forward by the logarithm map to the Lie algebra, and identify $f_{\fg}$ as a Lipschitz function on $\bR^q$.  
By \cite{A02a} Theorem 1.9.1, integrating the pointwise approximation to the heat kernel $u_N$,
\begin{equation}\label{alexopoulos}
 \left \langle L_g R_h f, \mu_{\varphi}^{*N}\right \rangle = \left \langle L_g R_h f, u_N\right \rangle + O\left(\|f\|_1 N^{-\frac{D+1}{2}} \right).
\end{equation}
Thus it suffices to show that
\begin{equation}\label{exchange_approximation}
 \left| \left\langle L_g R_h f, \mu^{*N}\right \rangle - \left\langle L_g R_h f, \mu_{\varphi}^{*N}\right \rangle \right| = O_\mu\left(\|f\|_1N^{-\frac{D+1}{2}} \right) + O_{\mu, A, f}\left( N^{-A}\right).
\end{equation}

As in Section \ref{test_function_section}, let $p_{g,h}$ and $q_{g,h}$ be polynomials such that 
\begin{equation} x' = \log g * x * \log h = p_{g,h}(x), \qquad x = q_{g,h}(x').\end{equation}
Thus
\begin{align}
 \left \langle L_g R_h f, \mu^{*N}\right \rangle &= \int_G f(gxh) d\mu^{*N}(x)\\
 \notag &= \int_{\fg} f_{\fg}\left(\log g * \Pi(\ux) *\log h \right) d\mu_{\fg}^{\otimes N} = \int_{\fg} f_{\fg}(p_{g,h}(\Pi(\ux)))d\mu_{\fg}^{\otimes N}. 
\end{align}

First consider the case that $\Ht(q_{g,h}) \geq N^{C}$ for a fixed constant $C$.  Let for some $B>0$, $\supp f_{\fg} \subset \left[ - \frac{B}{2}, \frac{B}{2}\right]^q$, let $\delta > 0$ and let
\begin{equation}
 S_{\Sm} = \left\{x' \in \left[ - \frac{B}{2}, \frac{B}{2}\right]^q:  \max_n \left\{\left\|q_{g,h}^{(n)}(x')\right\| N^{-\frac{n}{2}} \right\} \leq N^\delta \right\}.
\end{equation}
If $C$ is sufficiently large then Lemma \ref{small_values_lemma} implies that 
\begin{equation}
 \meas(S_{\Sm}) = O_{f, A}\left(N^{-A} \right).
\end{equation}
Let $f_{\fg} = f_{\fg, 1} + f_{\fg, 2}$, with $f_{\fg,1} = f_{\fg}|_{S_{\Sm}}$. Since $f$ is Lipschitz,
\begin{equation}
 \|f_1\|_1 \leq \|f\|_\infty \meas(S_{\Sm})
\end{equation}
and hence 
\begin{equation}\label{small_q}\langle L_g R_h f_1, \mu^{*N}\rangle, \;\langle L_g R_h f_1, \mu_{\varphi}^{*N}\rangle = O_{A, f}\left(N^{-A} \right).\end{equation}
Meanwhile, by Lemma \ref{truncation_lemma}, 
\begin{equation}
 \mu^{\otimes N}\left\{\max_n \left\{\left\|\Pi^{(n)}(x)\right\| N^{-\frac{n}{2}}\right\} > N^\delta\right\} = O_{\mu, A}\left(N^{-A}\right), 
\end{equation}
and similarly for $\mu_{\varphi}$.  Since $x' \in \supp f \setminus S_{\Sm}$ implies that 
\begin{equation}
 \max_n \left\{\left\|x^{(n)}\right\| N^{-\frac{n}{2}}\right\} > N^\delta
\end{equation}
it follows  that
\begin{equation}\label{large_q}
 \langle L_g R_h f_2,  \mu^{*N}\rangle,\; \langle L_g R_h f_2, \mu_{\varphi}^{*N}\rangle = O_{\mu, A, f}\left(N^{-A} \right).
\end{equation}
Together (\ref{small_q}) and (\ref{large_q}) imply Theorem \ref{main_theorem} in this case.

Now suppose that $\Ht(q_{g,h}) \leq N^C$ so that $\Ht(p_{g,h}) \ll N^{C'}$ for some $C' > 0$.
Let $\sigma$ be a compactly supported bump function on $\bR^q$ with dilation, for $t > 0$, $\sigma_t(x) = t^q \sigma\left(t x \right)$.  Let $f_{\fg, t} = f_{\fg} * \sigma_t$ be the Euclidean convolution.  Choose $t \gg \|f\|_1^{-1} N^{\frac{D+1}{2}}$ so that $\left\|f_{\fg} - f_{\fg, t}\right\|_\infty \ll \|f\|_1 N^{-\frac{D+1}{2}}$.  It thus suffices to prove (\ref{exchange_approximation}) with $f$ replaced by $f_t$. 

Expand, using the Fourier transform,
\begin{align}
\left \langle L_g R_h f_t, \mu^{*N} \right \rangle &= 
 \int_{\fg^N} f_{\fg, t}\left(\log g * \Pi(\ux) * \log h \right) d\mu_{\fg}^{\otimes N}(\ux)\\
\notag &= \int_{ \hat{\fg}} \widehat{L_g R_h f}_{\fg, t}(\xi) \E_{\mu_{\fg}^{\otimes N}}\left[e_\xi(\Pi(\ux)) \right] d\xi.
\end{align}
Since the  test function $f_{\fg, t}$ is smooth, the integral converges absolutely.  

Let $\epsilon_1 > \epsilon_2 > \cdots > \epsilon_s > 0$ be a collection of constants as in Lemmas \ref{frequency_truncation_lemma} and \ref{Lindeberg_lemma}.  Define
\begin{equation}
 E_{\Sm} = \left\{\xi \in \hat{\fg}: \max\left\{\left\|\xi^{(n)}\right\|N^{\frac{n}{2}-\epsilon_n}: 1 \leq n \leq s\right\}  \leq 1\right\}.
\end{equation}
Apply Lemma \ref{frequency_truncation_lemma}, and Lemma \ref{Fourier_decay_lemma} with $n = q+2$, to obtain
\begin{align}\label{fourier_truncation}
 &\left| \int_{ E_{\Sm}^c}  \widehat{L_g R_h f}_{\fg, t}(\xi) \chi_{N, \mu}(\xi) d\xi\right|\leq \left\|\chi_{N, \mu}\big|_{E_{\Sm}^c}\right\|_\infty \left\| \widehat{L_g R_h f}_{\fg, t}\big|_{E_{\Sm}^c}\right\|_1 = O_{\mu, A, \epsilon, f}\left(  N^{-A}\right).
\end{align}
Thus 
\begin{align}\label{truncated_representation}
 \left \langle L_g R_h f_t, \mu^{*N} \right \rangle = & O_{\mu, A, \epsilon,f}\left( N^{-A}\right)+ \int_{E_{\Sm}} \widehat{L_g R_h f}_{\fg, t}(\xi) \chi_{N, \mu}(\xi) d\xi,
\end{align}
and similarly with $\mu$ replaced by $\mu_{\varphi}$.

On the remainder of the integral, apply Lemma \ref{Lindeberg_lemma} to obtain
\begin{align}
 &\left|\int_{E_{\Sm}}\widehat{L_g R_h f}_{\fg, t}(\xi) \left(\chi_{N, \mu}(\xi) - \chi_{N, \mu_{\varphi}}(\xi)\right) d\xi\right|\\ \notag &\leq \|f\|_1 \left\|\left(\chi_{N, \mu}(\xi) - \chi_{N, \mu_{\varphi}}(\xi) \right)\big|_{E_{\Sm}}\right\|_\infty \meas(E_{\Sm}) \\& \notag \ll_\mu\|f\|_1 N^{-\frac{D+2}{2} + O(\epsilon_1)} .
\end{align}
Choose $\epsilon_1$ sufficiently small but fixed so that the error term is $O_\mu\left(\|f\|_1 N^{-\frac{D+1}{2}} \right)$, which proves Theorem \ref{main_theorem} in the remaining case.

Evidently the argument presented requires only finitely many moments of the measure $\mu$, but how many?  To gain convergence in the Fourier integral (\ref{fourier_truncation}) it was necessary to integrate by parts $n = q+2$ times, which costs a factor of $\Ht(p_{g,h})^{O_{s,q}(1)}$.  Hence the number of moments depends on $A$, the dimension $q$ and the step $s$, and hence is controlled by $A$ and the homogeneous dimension $D$.
\end{proof}

\bibliographystyle{plain}

\end{document}